\documentclass[11pt,a4paper]{amsart}
\usepackage[utf8]{inputenc}
\usepackage[english]{babel}
\usepackage{setspace}
\usepackage[T1]{fontenc}
\usepackage{amsmath}
\usepackage{amsfonts}
\usepackage{amssymb}
\usepackage{amsthm}
\usepackage{graphicx}
\usepackage{subfig}
\usepackage{float}
\usepackage{commath}
\usepackage{svg}
\usepackage{tikz-cd}
\usepackage{xcolor}
\usepackage{physics}
\usepackage{mathdots}
\usepackage{yhmath}
\usepackage{cancel}
\usepackage{siunitx}
\usepackage{array}
\usepackage{faktor}
\usepackage{cancel}
\usepackage{array}
\usepackage{multirow}
\usepackage[normalem]{ulem}

\usetikzlibrary{fadings}
\usetikzlibrary{patterns}
\usetikzlibrary{shadows.blur}
\usetikzlibrary{shapes}

\newcommand{\N}{\mathbb{N}}
\newcommand{\C}{\mathbb{C}}

\newtheorem{innercustomthm}{Theorem}
\newenvironment{customthm}[1]
  {\renewcommand\theinnercustomthm{#1}\innercustomthm}
  {\endinnercustomthm}
\newtheorem*{theorem*}{Theorem}
\newtheorem{theorem}{Theorem}[section]
\newtheorem{problem}{Problem}
\newtheorem{conjecture}{Conjecture}

\newtheorem{lemma}{Lemma}[section]
\newtheorem{prop}{Proposition}[section]
\theoremstyle{definition}
\newtheorem{definition}{Definition}[section]
\newtheorem{example}{Example}

\begin{document}

\title[Lipschitz geometry of finite mappings]
{Lipschitz geometry of the image of finite mappings}

\author[J.J. Nu\~no-Ballesteros]{Juan Jos\'e Nu\~no Ballesteros}
\author[V. de O. Prado]{Vin\'icius de Oliveira Prado}
\author[G. P. Sanchis]{Guillermo Pe\~nafort Sanchis}
\author[J.E. Sampaio]{Jos\'e Edson Sampaio}

\address{Jos\'e Edson Sampaio: Departamento de Matem\'atica, Universidade Federal do Cear\'a,
	      Rua Campus do Pici, s/n, Bloco 914, Pici, 60440-900, 
	      Fortaleza-CE, Brazil. E-mail: {\tt edsonsampaio@mat.ufc.br} 	   
}

\thanks{
The first named author was partially supported by the doctorate fellowship CAPES project number 88882.349957.\\
The second and third named authors were partially supported by Grant PID2021-124577NB-I00 funded by MCIN/AEI/10.13039/501100011033 and by ``ERDF A way of making Europe".\\
The last named author was partially supported by CNPq-Brazil grant 310438/2021-7 and by the Serrapilheira Institute (grant number Serra -- R-2110-39576).
}


\begin{abstract}
    This paper is devoted to the study of the LNE property in complex analytic hypersurface parametrized germs, that is, the sets that are images of finite analytic map germs from $(\C^n,0)$ to $(\C^{n+1},0)$. We prove that if the multiplicity of $f$ is equal to his generic degree, then the image of $f$ is LNE at 0 if and only if it is a smooth germ. We also show that every finite corank 1 map is sattisfies the previous hypothesis. Moreover, we show that for an injective map germ $f$ from $(\C^n,0)$ to $(\C^{n+1},0)$, the image of $f$ is LNE at 0 if and only if $f$ is an embedding.
\end{abstract}
\maketitle

\section{introduction}
Given an analytic set $X\subset \C^n$, we have two natural distances on it: (1) the outer distance, which is the restriction to $X\times X$ of the Euclidean distance of $\C^n$; (2) the inner distance, $d_X$, where $d_X(x,y)$ is given by the infimum of the length of the paths on $X$ connecting $x$ and $y$. The set $X$ is {\bf Lipschitz normally embedded} ({\bf LNE}) if there is a constant $C\geq 1$ such that $d_X(x,y)\leq C\|x-y\|$ for all $x,y\in X$. Given a point $p\in X$, the set $X$ {\bf Lipschitz normally embedded} ({\bf LNE}) {\bf at $p$} if there is an open neighbourhood $U$ of $p$ such that $X\cap U$ is LNE. 

The LNE notion was introduced by L. Birbrair and T. Mostowski \cite{BirbrairM:2000}, where they just call it normally embedded. As it was already remarked in \cite{NeumannPP:2020} and \cite{MendesS:2024}, Lipschitz normal embedding is a very active research area with many recent results giving necessary conditions for a set to be LNE in the real and complex setting, e.g., by A. Belotto da Silva, L. Birbrair, M. Denkowski, L. R. Dias, A. Fernandes, D. Kerner, R. Mendes, F. Misev, W. D. Neumann, J. J. Nu\~no-Ballesteros, H. M. Perdersen, A. Pichon, N. Ribeiro, M. A. S. Ruas, J. E. Sampaio, E. C. da Silva, M. Tibar, K. Zajac etc \cite{BelottoFP:2022,BirbrairDMS:2025,BirbrairMN-B:2018,DenkowskiT:2019,DenkowskiZ:2024,DiasR:2022,FernandesS:2019,FernandesS:2020,FernandesS:2023,FernandesS:2024,KernerPR:2018,MisevP:2021,Nhan:2023,NeumannPP:2020,NeumannPP:2020b,Sampaio:2020,Sampaio:2023,Sampaio:2024,SampaioS:2025}.

It is natural to try characterizing the LNE property, but this is a hard problem in general. Consequently, it is common to restrict this problem to some special class of analytic sets. An example of a class of analytic sets where the LNE property is well understood is the class of analytic curves. An important class that, from the Lipschitz geometry point of view, generalizes the (irreducible) analytic curves is that formed by the analytic sets that are images of finite holomorphic mappings $f\colon (\C^n,0)\to (\C^{n+1},0)$, the so-called parametrized hypersurfaces. Here, we restrict our attention to characterizing LNE on a larger class of hypersurfaces, namely to the images of finite map-germs $f\colon (X^n,0)\to (\C^{n+1},0)$, where $X^n$ is a $n$-dimensional Cohen-Macaulay space. 

Let $f\colon (X^n,0)\to (\C^{n+1},0)$ be such a map-germ. If $X$ is irreducible, then so is the image of $f$, and we may define the \textit{degree} of $f$, denoted by \[{\deg f,}\] as the number of preimages of a generic point in the image of $f$. Regardless of the irreducibility of $X$, the \textit{multiplicity} of $f$ is the number
    \[\operatorname{mult} f=\dim_{\C}\mathcal{O}_X/\langle f\rangle,\]where $\langle f\rangle$ stands for the ideal generated by the coordinate functions of $f$. The  \textit{generic degree} of $f$, denoted by 
    \[\operatorname{gd}(f),\]
is the minimum of the dimensions of the $\C$-vector spaces $\mathcal{O}_X/\langle \pi\circ f\rangle$, where $\pi$ ranges through all linear projections onto $\C^n$. This numbers satisfy the inequalities
    \[\operatorname{deg} f\leq \operatorname{mult} f\leq\operatorname{gd}(f) \]

Observe that the following statements are equivalent:
\begin{itemize}
\item $\operatorname{mult} f=\operatorname{gd}(f)$
\item After a linear change of coordinates, the map-germ takes the form $f=(\overline f,h)$, with $h\in \langle \overline f\rangle$.
\end{itemize}

Our first result is

\begin{customthm}{\ref*{thm:caso geral}}
    Let $f\colon (X^n,0)\to (\C^{n+1},0)$ be a finite map germ such that $(X,0)$ is smooth and $\operatorname{mult}f=\operatorname{gd}(f)$. Then the following items are equivalent
    \begin{itemize}
        \item [(a)] $f(X,0)$ is LNE;
        \item[(b)] $f(X,0)$ is smooth;
        \item[(c)] $\operatorname{mult}f=\deg f.$
    \end{itemize}
\end{customthm}

The previous result has couple particular cases which are interesting on their own. The first one is as follows:

\begin{customthm}
     {\ref*{teo:caso corank 1 finita qualquer}}
     Let $f\colon (\C^n,0)\to (\C^{n+1},0)$ be a finite corank $\leq$ 1 holomorphic map germ. Then the following conditions are equivalent
     \begin{itemize}
         \item[(a)] $f(\C^n,0)$ is LNE;
         \item[(b)] $f(\C^n,0)$ is smooth;
         \item[(c)] $\text{mult }f=\deg f$.
     \end{itemize}
 \end{customthm}

The other interesting particular case is related  to the following problem:
\begin{problem}
Let $\varphi \colon (\C^n,0)\to (\C^{n+1},0)$ be an injective holomorphic mapping. What additional conditions on $\varphi$ guarantee for it to be an embedding?
\end{problem}
Our result gives a contribution of Lipschitz geometry to the above problem.
\begin{customthm}
     {\ref*{teo:caso injetivo}}
     Let $f\colon (\C^n,0)\to (\C^{n+1},0)$ be an injective holomorphic map germ. Then the following conditions are equivalent:
     \begin{itemize}
         \item[(a)] $f(\C^n,0)$ is LNE;
         \item[(b)] $f(\C^n,0)$ is smooth;
         \item[(c)] $f$ is a holomorphic embedding.
     \end{itemize}
 \end{customthm}
Note that the above theorem is related to the following very interesting problem.
\begin{conjecture}[L\^e's conjecture]
\label{conjectura de Le}
Let $f\colon (\C^n,0)\to (\C^{n+1},0)$ be an injective holomorphic map germ. Then $f$ is a corank 1 map germ. 
\end{conjecture}
In fact, if L\^e's conjecture has a positive answer, then the above theorem follows from Theorem \ref{teo:caso corank 1 finita qualquer}. In any case, our proof of Theorem \ref{teo:caso injetivo} does not depend on L\^e's conjecture.

\section{Preliminary definitions and results}
The results in this section can be found in \cite{FernandesS:2019} and \cite{birbrair2016regularity}.

\begin{definition}
    Let $X\subset \C^n$ an analytic set and $x\in X$, the \textit{tangent cone of $X$ in $x$ is the set}
    $$C(X,x)=\left\{v\in \C^n\colon \exists \{x_n\}_{n\in\N}\subset X,\{t_n\}_{n\in\N}\subset \C; x_n\to x \land t_nx_n\to v \right\}.$$
\end{definition}

\begin{theorem}
    If $X\subset \C^n$ is analytic and $x\in X$, let $\mathcal{I}(X,x)$ be the ideal of $X$ in $x$. Then 
    $$C(X,x)=V(\mathcal{I_*}(X,x)),$$
    where $\mathcal{I_*}(X,x)$ is the ideal generated by the initial parts of all functions in $\mathcal{I}(X,x)$.
\end{theorem}

\begin{theorem}
    \label{teo:regularidade LNE}
    If $X\subset \C^n$ is an analytic set that is LNE at $0$ and $C(X,0)$ is a linear subspace, then $X$ is smooth at $0$.
\end{theorem}

In the whole text, we will always assume that if $f\colon (X^n,0)\to (\C^{p},0)$ is an analytic map germ, then $(X,0)$ is a Cohen-Macaulay analytic space of dimension $n$.

\begin{definition}
    We say that $f,g\colon (\C^n,0)\to (\C^p,0)$ are $\mathcal{A}$-equivalent if exists $\phi\colon (\C^n,0)\to (\C^n,0)$ and $\varphi\colon(\C^p,0)\to (\C^p,0)$ diffeomorphisms such that $g=\varphi\circ f\circ \phi^{-1}$.
\end{definition}
Since tangent cones and the LNE property are invariant under diffeomorphisms, our calculations must only be performed in representants of $\mathcal{A}$-equivalent classes.
\section{Calculating tangent cones by the Mond-Pellikaan algorithm}
Let $f\colon (X^n,0)\to (\C^{n+1},0)$ be a finite analytic map germ, writing $f(z)=(\overline{f}(x),h(x))$, with $\overline{f}(x)=(f_1(x),...,f_n(x))$ being the first $n$ coordinates and $h\in \mathcal{O}_X$. There exists a linear change of coordinates in the target turning $f$ into a mapping for which $\overline f$ is finite. Therefore, up to $\mathcal{A}$-equivalence, we may assume for $\overline{f}$ to be finite. We write $\mathcal{O}_X$ for the ring of function germs in the source of $\overline{f}$ and $\mathcal{O}_{n+1}$ for the ring of function germs in the target of $\overline{f}$. The ring $\mathcal{O}_X$ can be regarded as an $\mathcal{O}_{n+1}$-module with the following product rule: for $g\in \mathcal{O}_{n+1}$ and $a\in \mathcal{O}_X$
\[g(\mathcal{X})a(x):= g\circ f(x)a(x).\]
This push-forward $\mathcal{O}_T$-module will be denoted by $\overline{f}_*\mathcal{O}_X$. By Nakayama's lemma, taking representatives of the elements of a basis of the $\C$-vector space $\mathcal{O}_X/\langle f_1,...,f_n \rangle$, which has finite dimension since $\overline{f}$ is finite, one obtains a minimal set of generators of $\overline{f}_*\mathcal{O}_X$, say \[g_1,\dots, g_k.\]
Writing the functions $g_1h,g_2h,...,g_kh$ in terms of this generating set, one obtains expressions
$$g_jh(x)=\sum_{i=1}^ka_{ij}(\mathcal{X})g_i(x).$$

The matrix $\lambda=\left(\{a_{ij}(\mathcal{Z})\}_{k\times k}-\mathcal{Z}_{n+1}Id_{k\times k}\right)$ is a presentation matrix for $f_*\mathcal{O}_X$ (see \cite{mond2006fitting}). The function $\lambda$ is the generator of the $0^{th}$ Fitting Ideal of $f$, denoted by $\mathcal{F}_0(f)$. Therefore, $f(X,0)=V(\lambda)$ (see Corollary 11.2 in \cite{mondsingularitiesmappings}).

Notice now that the generators $g_1,\dots, g_k$ can always be chosen as monomials since $\mathcal{O}_X/\langle \overline{f}\rangle$ is a finite dimensional $\C$-vector space and the set of all classes of monomials generates this quotient, so there is a subset of the monomials that forms a basis for the vector space. Moreover, one of the terms of the basis is always in the class of $1$ since $\overline{f}(0)=0$, hence we may set $g_1=1$.

\begin{definition}
    Let $f\colon (X^n,0)\to (\C^{n+1},0)$ be a finite map germ whose image $f(X)$ is irreducible. The \textit{degree} of $f$, denoted by $\deg f$, is the number of generic preimages of $f$.
\end{definition}

\begin{prop}
\label{prop:determinante da a equação}
    Let $f\colon (X^n,0)\to (\C^{n+1},0)$ a finite map germ, then $\mathcal{F}_0(f)$ is generated by $h^k$, where $h\in \mathcal{O}_{n+1}$ is the reduced equation that generates the image of $f$ and $k$ is the degree of $f$ onto its image.
\end{prop}
\begin{proof}
    See Proposition 11.7 in \cite{mondsingularitiesmappings}.
\end{proof}

Due to the previous proposition, if we want to see who is the tangent cone as a set, we can just take the initial part $\lambda_*$ of $\lambda=h^k$.

\begin{example}
   Let $f(x,y)=(x,y^2,xy)$.

    In this example, $\overline{f}(x,y)=(x,y^2)$ and the generators of $\overline{f}_*\mathcal{O_2}$ are $\{1,y\}$, and we obviously have
    $$xy=0\cdot 1+\mathcal{X}\cdot y$$
    $$y(xy)=\mathcal{XY}\cdot 1+0\cdot y$$
    Then the equation that defines $f(\C^2,0)$ is $\lambda =0$, where
    $$\lambda=\det \begin{bmatrix}
        -z&xy\\
        x&-z
    \end{bmatrix}=z^2-x^2y$$
    and then the tangent cone is $\{z^2=0\}$.
\end{example}

\begin{example}
\label{ex: y^2+função}
\text{Consider an analytic map germ of the form }  \[f(x,y)=(x,y^2+g(x,y),h(x,y)),\]  with $g\in \langle x,y^2 \rangle$.
    
    In this case, $\overline{f}(x,y)=(x,y^2+g(x,y))$, so $\{1,y\}$ are the generators of $\overline{f}_*\mathcal{O}_2$. Write $\mathcal{X}(x,y)=x$ and $\mathcal{Y}(x,y)=y^2+g(x,y)$ for the target coordinates. Now, let $r_1,r_2,p,q\in \mathcal{O}_2$ such that
    $$y^2= r_1(\mathcal{X},\mathcal{Y})\cdot 1+r_2(\mathcal{X},\mathcal{Y})\cdot y$$
    and
    $$h(x,y)=p(\mathcal{X},\mathcal{Y})\cdot 1+q(\mathcal{X},\mathcal{Y})\cdot y.$$
    Therefore, we have that
    \begin{align*}
        y\cdot h(x,y)&=p(\mathcal{X},\mathcal{Y})\cdot y+q(\mathcal{X},\mathcal{Y})\cdot y^2\\
        &=r_1(\mathcal{X},\mathcal{Y})q(\mathcal{X},\mathcal{Y})\cdot 1+ (p(\mathcal{X},\mathcal{Y})+r_2(\mathcal{X},\mathcal{Y})q(\mathcal{X},\mathcal{Y}))\cdot y.
    \end{align*}
    Then, the equation that defines $f(\C^2,0)$ is given by the vanishing of
    \begin{align*}
        \lambda&=\det\begin{bmatrix}
            p(x,y)-z&r_1(x,y)q(x,y)\\
            q(x,y)&p(x,y)+r_2(x,y)q(x,y)-z
        \end{bmatrix}\\
        &=(p(x,y)-z)(p(x,y)-z+r_2(x,y)q(x,y))-r_1(x,y)q^2(x,y)\\
        &=(p(x,y)-z)^2+r_2(x,y)q(x,y)(p(x,y)-z)-r_1(x,y)q^2(x,y).
    \end{align*}

    Applying $(0,0)$ in the equation of $y^2$, we get that $ord(r_1)\geq 1$, and since $y^2$ has order 2, then $ord(r_2)\geq 1$, because if $r_2(0,0)=C\neq 0$, then we would have a term $Cy$ in the series of $y\cdot r_2(\mathcal{X},\mathcal{Y})$, which cannot be cancelled with a term of $r_1(\mathcal{X},\mathcal{Y})$.

    Furthermore, applying $(0,0)$ to the equation of $h(x,y)$, we get $ord(p)\geq 1$, and since $f$ is singular, we have that $\frac{\partial h}{\partial y}(0,0)= 0$, but
    $$0=\frac{\partial h}{\partial y}(0,0)=q(0,0),$$
    then $ord(q)\geq 1$.

    Therefore, the second and third therms of the sum in $\lambda$ are in $\mathfrak{m}^3$, then the equation that defines the tangent cone of $f(\C^2,0)$ is 
    $$\lambda_*=\left\{ \begin{array}{lr}
       (p(x,y)_*-z)^2  & ; ord(p)=1 \\
        z^2 & ; ord(p)>1
    \end{array}\right.$$
    Therefore, $C(X,0)$ is a linear subspace.
\end{example}

Now, let us give the primary definitions to do the main result, in both of them, let $f\colon (X^n,0)\to (\C^{n+1},0)$ be a finite map germ.

\begin{definition}
    The \textit{multiplicity} of $f$ is the number
    $$\operatorname{mult} f=\dim_{\C}\mathcal{O}_X/\langle f\rangle.$$
\end{definition}
\begin{definition}
    The \textit{generic degree} of $f$, denoted by $\operatorname{gd}(f)$, is the minimum number
    $$\dim_{\C}\mathcal{O}_X/\langle \pi\circ f\rangle$$
    where $\pi$ is a linear projection of $\C^{n+1}$ in $\C^{n}$.
\end{definition}

It is clear that $\operatorname{mult} f\leq \operatorname{gd}(f)$, and since $\operatorname{gd}(f)=deg(\pi\circ f)$ for any such $\pi$ that minimizes the dimension, then $\deg(f)\leq \operatorname{gd(f)}$. After a linear change of coordinates in the codomain, we can suppose that $\pi$ is the projection onto the first $n$ coordinates, so, $\operatorname{gd}(f)=\dim \mathcal{O}_X/\langle \overline{f}\rangle=\deg \overline{f}$, that will be the size of the presentation matrix given by the algorithm.

Given $g_1,\dots,g_k$ is a minimal set of generators of $\overline{f}_*\mathcal{O}_X$ just like described in the algorithm. Then for each $i,j=1,\dots,k$, write
$$g_ig_j(x)=\sum_{s=1}^k \Gamma^{ij}_s(\mathcal{X})\cdot g_s(x).$$
The next Lemma will show that the $g_i$ can be chosen smartly so that the presentation matrix will have a good behaviour when $(X,0)$ is smooth.

\begin{lemma}
    \label{lemma:ordenação base}
    Let $\overline{f}\colon (X^n,0)\to (\C^{n},0)$ be a finite map germ with $(X,0)$ smooth. Then there is a ordered minimal set of generators of $\overline{f}_*\mathcal{O}_X$ such that for every $i,j\neq 1$, we have $\Gamma^{ij}_s\in \mathfrak{m}$ for every $s\leq \max\{i,j\}$.
\end{lemma}
\begin{proof}
    As previously mentioned, we may initially take $g_1,\dots,g_k$ to be a monomial set of generators for $\overline{f}\mathcal{O}_X$. After possibly changing the order of the $g_i$'s, we may assume that if $\deg g_i<\deg g_j$, then $i<j$, that is, the ordering is by polynomial degree. We will show that this ordered set of generators is the right one.
    
    Suppose that, for some $i,j\neq 1$, we have $\Gamma^{ij}_s\notin \mathfrak{m}$ for some $s\leq i$, so $\Gamma^{ij}_s(\mathcal{X})g_s(x)=Cg_s(x)+A(x)$, where $A(0)=0$. Since $\partial g_s<\partial g_ig_j$ because $g_s$ has degree at most the degree of $g_i$, we have that $Cg_s(x)$ must be appear in the expression 
    \begin{align*}
        &\sum_{l\neq s}\Gamma^{ij}_l(\mathcal{X})g_l(x)\\
        =&\sum_{l\neq s}\Gamma^{ij}_l(0)g_l(x)+\sum_{l\neq s}(\Gamma^{ij}_l(\mathcal{X})-\Gamma^{ij}_l(0))g_l(x).
    \end{align*}
    The fist sum cannot contain a $Cg_s(x)$ because it is a polynomial with monomials $g_l$ that are different from $g_s$. So $Cg_s(x)$ must appear in the second sum, that is,
    $$\sum_{l\neq s}(\Gamma^{ij}_l(\mathcal{X})-\Gamma^{ij}_l(0))g_l(x)=-Cg_s(x)+B(x).$$
    Passing to the quotient $\mathcal{O}_X/\langle\overline{f}\rangle$, we get a contradiction since the left side is in the class of $0$ but the right side is not.
\end{proof}

\textbf{Remark.} The equality $\operatorname{mult} f=\operatorname{gd}(f)$ is equivalent to the statement that, after a linear change of coordinates, $f=(\overline{f},h)$ with $h\in \langle\overline{f}\rangle$. This is because
$$(\langle\overline{f}\rangle+\langle h\rangle)/\langle\overline{f}\rangle\rightarrow \mathcal{O}_X/\langle\overline{f}\rangle\rightarrow \mathcal{O}_X/(\langle\overline{f}\rangle+\langle h\rangle$$
is exact, so $\operatorname{gd}(f)=\operatorname{mult} f+\dim_{\C}\langle f\rangle/\langle\overline{f}\rangle$.

Now, we are ready to prove the main result

\begin{theorem}
    \label{thm:caso geral}
    Let $f\colon (X^n,0)\to (\C^{n+1},0)$ be a finite map germ such that $(X,0)$ is smooth and $\operatorname{mult}f=\operatorname{gd}(f)$. Then the following items are equivalent
    \begin{itemize}
        \item [(a)] $f(X,0)$ is LNE;
        \item[(b)] $f(X,0)$ is smooth;
        \item[(c)] $\operatorname{mult}f=\deg f.$
    \end{itemize}
\end{theorem}
\begin{proof}
    Write $f=(\overline{f},h)$ where $\dim_{\C}\mathcal{O}_X/\langle\overline{f}\rangle=\operatorname{gd}(f)$ and let $1=g_1,g_2,\dots,g_k$ be a minimal set of generators for $\overline{f}_*\mathcal{O}_X$ as in Lemma \ref{lemma:ordenação base}. Since $\operatorname{mult}f=\operatorname{gd}(f)$, then $k=\operatorname{mult}f$. Now write 
    $$h=\sum_{i=1}^k p_i(\mathcal{X})g_i(x).$$
    Calling $\mathcal{Z}=h$ the $n+1^{th}$ coordinate in the codomain and doing the diffeomorphic coordinate change $(\mathcal{X},\mathcal{Z})\mapsto (\mathcal{X},\mathcal{Z}-p_1(\mathcal{X}))$, we can suppose that $p_1\equiv 0$.
    Now multiply $h$ by each $g_j$, we get
    \begin{align*}
        g_jh(x)&=\sum_{i=2}^k p_i(\mathcal{X})g_i(x)g_j(x)\\
        &=\sum_{i=2}^k p_i(\mathcal{X})\left(\sum_{s=1}^k \Gamma^{ij}_s(\mathcal{X})\cdot g_s(x)\right)\\
        &=\sum_{s=1}^k \underbrace{\left(\sum_{i=2}^kp_i(\mathcal{X})\Gamma^{ij}_s(\mathcal{X})\right)}_{P_s^j(\mathcal{X})}g_s(x)
    \end{align*}
    Since $h\in \langle\overline{f}\rangle$, it is $0$ in the quotient. Then each $p_i$ is in $\mathfrak{m}$. So, every entry of the presentation matrix is in $\mathfrak{m}$, our aim is now to prove that the entries above the diagonal are in fact in $\mathfrak{m}^2$ and that the only terms of degree 1 that appears in the diagonal are the new variable $-\mathcal{Z}$. If $s\leq j$, then
    \begin{align*}
        P_s^j(\mathcal{X})=\sum_{i=2}^k\underbrace{p_i(\mathcal{X})}_{\in \mathfrak{m}}\underbrace{\Gamma^{ij}_s(\mathcal{X})}_{\in \mathfrak{m}}
    \end{align*}
    So, $P_s\in \mathfrak{m}^2$ for all $s\leq j$, this says that the presentation matrix will be
    \[
\left[
\begin{array}{ccccc}
-\mathcal{Z}+P_1^1(\mathcal{X}) &        &        &        &  \\
       &  &        &   \text{\LARGE{A}}     &   \\
       &        & \ddots &        &   \\
   \text{\LARGE{B}}    &        &        & -\mathcal{Z}+P_k^k(\mathcal{X}) &   \\
      &        &        &        & 
\end{array}
\right]
\]
This way, we can found the tangent cone of $f(X,0)$ by taking the determinant of the presentation matrix and then taking the initial part of it. Notice however that since the determinant is a sum of products calculated by means of permutations, let us do each permutation at a time.

First, the trivial permutation will gives, in the determinant, the product of all the diagonal terms, that is,
$$\prod_{i=1}^k \left(-\mathcal{Z}+P_i^i(\mathcal{X})\right),$$
since $P_i^i\in \mathfrak{m}^2$, the initial part of this product is just $(-\mathcal{Z})^k$. Now pick any non-trivial permutation $\sigma\in S_k$, there are two numbers $u,v$ such that $\sigma(u)< u$ and $\sigma(v)>v$. In other words, any non-trivial permutation in the determinant will contain a multiplication of a term contained in the upper triangle $A$ in the presentation matrix. So, the order of all the other summed terms in the determinant is at least $k+1$ since we are multiplying $k$ factors in $\mathfrak{m}$ and at least one of them is in $\mathfrak{m}^2$. Therefore, this terms will not appear in the initial part, showing that the tangent cone $C(f(X),0))$ is a hyperplane. This together with Theorem \ref{teo:regularidade LNE} and the fact that every smooth germ is LNE shows the equivalence $(a)\Leftrightarrow (b)$.

Now, to show the equivalence $(b)\Leftrightarrow(c)$, notice that we got that the initial part of the determinant of the presentation matrix is $(-\mathcal{Z})^k$, where $k=\operatorname{mult}f$. On the other hand, by Proposition \ref{prop:determinante da a equação}, the determinant is $F^{\deg f}$, where $F$ is a reduced equation such that $f(X,0)=V(F)$. We know that $f(X,0)$ is smooth if, and only if $F$ have order $1$, and this occur if, and only if $\deg f=k$ because the order of $F^{\deg f}=k$. Therefore, $f(X,0)$ is smooth if, and only if $\operatorname{mult}f=\deg f$.
\end{proof}

The aim now is to see if the corank of the map gives information about the LNE property.

\begin{lemma}
\label{lemma:corank 1}
    If $f\colon (\C^n,0)\to (\C^{n+1},0)$ is a finite corank 1 map germ, then $f$ is $\mathcal{A}$-equivalent to a map of the form $(x_1,\dots, x_{n-1},y^m+g(x,y),h(x,y))$ that satisfies the hypotheses of Theorem \ref{thm:caso geral}.
\end{lemma}
\begin{proof}
    We will show that the last coordinate of $f$ is in the ideal generated by the others, that is equivalent to say $\operatorname{mult}f=\operatorname{gd}(f)$ as remarked before. Write the coordinates of $\C^n$ as $(x,y)\in \C^{n-1}\times \C$. By the Rank Theorem and since $f$ has corank 1, $f$ is $\mathcal{A}$-equivalent to a map of the form $(x,p(x,y),h(x,y))$. Since $f$ is finite, either $\text{ord}_yp<\infty$ or $\text{ord}_yh<\infty$, this is because if both orders are $\infty$, then $F(0,y)=(0,p(0,y),h(0,y))=(0,0,0)$ and then $f$ would not be finite. Suppose without loss of generality that $\text{ord}_yp\leq \text{ord}_yh$. There are two cases
    
    \textbf{Case 1.} ($\text{ord}_yh=t<\infty$). By the Weierstrass Preparation Theorem, we can write $g$ and $h$ as
    $$p(x,y)=r_1(x,y)\left(y^m+a_1(x)y^{m-1}+\dots a_{m-1}(x)y+a_m(x)\right)$$
    and 
    $$h(x,y)=r_2(x,y)\left(y^t+b_1(x)y^{t-1}+\dots b_{t-1}(x)y+b_t(x)\right),$$
    where $r_1,r_2$ are units and $a_i(0,0)=b_j(0,0)=0$ for each $0\leq i\leq m-1$, $0\leq j\leq t-1$ (We will still call this map $f$).
    
    It is clear that $\langle x,p(x,y)\rangle=\langle x,y^m\rangle$, so $\{1,y,\dots,y^{m-1}\}$ forms a basis for $\mathcal{O}_X/\langle \overline{f}\rangle$ as a $\C$-vector space. Furthermore, writing $h$ in terms of the generators
    $$r_2(x,y)\left(y^t+b_1(x)y^{t-1}+\dots b_{t-1}(x)y+b_y(x)\right)=\sum_{i=0}^{m-1}P_i(\mathcal{X,Y})y^i,$$
    we get that $P_i\in \mathfrak{m}$ for each $0\leq i\leq m-1$ because $P_i(0,0)\neq 0$ implies a free term in $y$ with degree less then $m$, which is a contradiction with the fact that $m\leq t$. And this finishes the proof for this case.

    \textbf{Case 2.} ($\text{ord}_yh=\infty$). In this case, we can also write 
    $$p(x,y)=r(x,y)\left(y^m+a_1(x)y^{m-1}+\dots a_{m-1}(x)y+a_m(x)\right)$$
    where $r$ is a unit and $a_i(0)=0$ for each $0\leq i\leq m-1$. Writing $h$ in the terms of the generators
    $$h(x,y)=\sum_{i=0}^{m-1}P_i(\mathcal{X,Y})y^i,$$
    we claim that $P_i\in \mathfrak{m}$. In fact, 
    \begin{align*}
        0\equiv h(0,y)&=\sum_{i=0}^{m-1}P_i(0,p(0,y))y^i\\
        &=\sum_{i=0}^{m-1}P_i(0,r(0,y)y^m)y^i
    \end{align*}
    Notice that $P_i(0,p(0,y))y^i$ have a term of degree $i$ if $P_i\notin \mathfrak{m}$ and all other terms will be of degree at least $m+i$. So, if some $P_i\notin \mathfrak{m}$, then a term of degree $i$ would appear in the expression of the function $0$, which is a contradiction. And this finishes the proof for this case as well.
\end{proof}

\textbf{Remark.} In the above lemma, the multiplicity of $f$ is exactly $m$, the exponent of $y$ in $p(x,y)$.

Now we can characterize the LNE property in every image of finite corank 1 map germs. Putting together Theorem \ref{thm:caso geral} and Lemma \ref{lemma:corank 1}, we get what follows:

\begin{theorem}
     \label{teo:caso corank 1 finita qualquer}
     If $f\colon (\C^n,0)\to (\C^{n+1},0)$ is a finite corank 1 map germ, then the following conditions are equivalent
     \begin{itemize}
         \item[(a)] $f(\C^n,0)$ is LNE;
         \item[(b)] $f(\C^n,0)$ is smooth.;
         \item[(c)] $\text{mult }f=\deg f$.
     \end{itemize}
 \end{theorem}

\begin{example}
    Let $f\colon (\C^2,0)\to (\C^3,0)$ given by $f(x,y)=(x,y^k,h(x,y^k))$ for $k>1$, the image of $f$ is generated by $z-h(x,y)=0$, which is smooth, $f$ is not an embedding and $\operatorname{mult} f=k=\deg f$.
\end{example}

It is also easy to find an example of a corank 2 map with $\operatorname{mult} f>\deg f$ but with LNE image.

\begin{example}
    Let $f\colon (\C^2,0)\to (\C^{3},0)$ be the map given by $f(x,y)=(x^2,y^2,xy)$, the algorithm gives the equation $(z^2-xy)^2$. Therefore, the reduced equation for his tangent cone is $z^2-xy=0$, which is not a linear subspace. On the other hand, $\operatorname{mult}f=3>2=\deg f$ and $f(\C^n,0)$ is LNE. This example shows that Theorem \ref{teo:caso corank 1 finita qualquer} does not pass through corank 2 map germs in general, showing that the hypothesis of $\operatorname{mult}f=\operatorname{gd}(f)$ is not superfluous.
\end{example}

\section{Injective case}
Assuming that Lê's conjecture \ref{conjectura de Le} is true, Theorem \ref{teo:caso corank 1 finita qualquer} already includes the red{ two dimensional injective case, since then injectivity forces $\operatorname{corank} f=1$}, but we will show that, in fact, the equivalence holds for all injective map germs even if the conjecture is false.

For corank 1 mappings, the singular locus of $f(\C^n,0)$ is easier to find; in particular, if $f$ is injective, we can drop the corank 1 hypothesis and get a more general statement. Let us first state some necessary results.

\begin{prop}
\label{prop: conjunto singular da imagem}
    Let $f\colon (\C^n,0)\to (\C^{n+1},0)$ be finite and generically 1-to-1. Then the singular locus of $f(\C^n,0)$ is the set of points with at least two preimages, counting multiplicity. In other words, the singular locus of $f(\C^n,0)$ is the set formed by images of non-injective points and images of critical points.
\end{prop}
\begin{proof}
This follows from Corollary 11.2 and Proposition 11.10 in \cite{mondsingularitiesmappings}.
\end{proof}

\begin{prop}[Corollary 11.5 in \cite{mondsingularitiesmappings}]
\label{prop: dimensão conjunto singular}
    Let $f\colon (\C^n,0)\to (\C^{n+1},0)$ be finite and generically 1-to-1. Then the singular locus of $f(\C^n,0)$ has dimension $n-1$.
\end{prop}

\begin{prop}
\label{prop:dimensão pontos de corank grande}
    Let $f\colon (\C^n,0)\to (\C^{n+1},0)$ be a finite map and define
    $$\hat{\Sigma}^k(f)=\{x\in \C^n\colon \text{corank }df_x\geq k\}.$$
    Then $\dim \hat{\Sigma}^k(f)\leq n-k$.
\end{prop}
\begin{proof}
    See for example Lemma 2.8 in \cite{ballesteros2017multiplepoints}.
\end{proof}
With these three results, we can prove the injective case.

\begin{theorem}\label{teo:caso injetivo}
     Let $f\colon (\C^n,0)\to (\C^{n+1},0)$ be an injective holomorphic map germ. Then the following conditions are equivalent
     \begin{itemize}
         \item[(a)] $f(\C^n,0)$ is LNE;
         \item[(b)] $f(\C^n,0)$ is smooth;
         \item[(c)] $f$ is a holomorphic embedding.
     \end{itemize}
 \end{theorem}
\begin{proof}
    If $f$ is an embedding, then $f(\C^n,0)$ is smooth and then obviously LNE. Suppose now that $f$ is singular and $X$ is LNE, we know by Propositions \ref{prop: conjunto singular da imagem} and \ref{prop: dimensão conjunto singular} that $\text{sing}X$ is formed by the images of non-immersive points of $f$ and some strict multiple points, and that it has dimension $n-1$. But, since $f$ is injective, then $\text{sing}X=\{f(x)\colon \text{corank }df_x>0\}$. By Proposition \ref{teo:caso corank 1 finita qualquer}, the image of a singular injective corank 1 map cannot be LNE, so $\text{sing}X=\{f(x)\colon \text{corank }df_x\geq 2\}$. On the other hand, Proposition \ref{prop:dimensão pontos de corank grande} guarantees that this set is the image of a set with dimension at most $n-2$, which contradicts the fact that $\dim\text{sing}X=n-1$. 
\end{proof}

\end{document}